\documentclass[12pt]{article}
\usepackage{amssymb}
\usepackage{amsthm}
\usepackage{amsmath}
\usepackage{graphicx}
\usepackage{enumerate}

\newtheorem{theorem}{Theorem}[section]
\newtheorem{definition}[theorem]{Definition}
\newtheorem{lemma}[theorem]{Lemma}
\newtheorem{proposition}[theorem]{Proposition}
\newtheorem{remark}[theorem]{Remark}
\newtheorem{example}[theorem]{Example}
\newtheorem{corollary}[theorem]{Corollary}
\title{Adjoint operations in twist-products of lattices}
\author{Ivan~Chajda and Helmut~L\"anger}
\date{}
\begin{document}

\footnotetext{Support of the research by the Austrian Science Fund (FWF), project I~4579-N, and the Czech Science Foundation (GA\v CR), project 20-09869L, entitled ``The many facets of orthomodularity'', as well as by \"OAD, project CZ~02/2019, entitled ``Function algebras and ordered structures related to logic and data fusion'', and, concerning the first author, by IGA, project P\v rF~2020~014, is gratefully acknowledged.}

\maketitle

\begin{abstract}
Given an integral commutative residuated lattice $\mathbf L=(L,\vee,\wedge)$, its full twist-product $(L^2,\sqcup,\sqcap)$ can be endowed with two binary operations $\odot$ and $\Rightarrow$ introduced formerly by M.~Busaniche and R.~Cignoli as well as by C.~Tsinakis and A.~M.~Wille such that it becomes a commutative residuated lattice. For every $a\in L$ we define a certain subset $P_a(\mathbf L)$ of $L^2$. We characterize when $P_a(\mathbf L)$ is a sublattice of the full twist-product $(L^2,\sqcup,\sqcap)$. In this case $P_a(\mathbf L)$ together with some natural antitone involution $'$ becomes a pseudo-Kleene lattice. If $\mathbf L$ is distributive then $(P_a(\mathbf L),\sqcup,\sqcap,{}')$ becomes a Kleene lattice. We present sufficient conditions for $P_a(\mathbf L)$ being a subalgebra of $(L^2,\sqcup,\sqcap,\odot,\Rightarrow)$ and thus for $\odot$ and $\Rightarrow$ being a pair of adjoint operations on $P_a(\mathbf L)$. Finally, we introduce another pair $\odot$ and $\Rightarrow$ of adjoint operations on the full twist-product of a bounded commutative residuated lattice such that the resulting algebra is a bounded commutative residuated lattice satisfying the double negation law and we investigate when $P_a(\mathbf L)$ is closed under these new operations $\odot$ and $\Rightarrow$.
\end{abstract}

{\bf AMS Subject Classification:} 06D30, 03G10, 03G25, 03G47

{\bf Keywords:} Full twist-product, residuated lattice, Kleene lattice, pseudo-Kleene lattice, double negation law

\section{Introduction}

Kleene lattices were introduced by J.~A.~Kalman (\cite K) (under a different name) as a special kind of De Morgan lattices which serve as an algebraic axiomatization of a certain propositional logic satisfying the double negation law but not necessary the excluded middle law. If the underlying lattice is not distributive such lattices are called pseudo-Kleene (see e.g.\ \cite{Ch}). It is a question if certain binary operations can be introduced in a Kleene or pseudo-Kleene lattice such that they form an adjoint pair. To solve this problem, we apply an approach using the full twist-product construction and another construction extending a distributive lattice into a Kleene one.

Having a residuated lattice $(L,\vee,\wedge,\cdot,\rightarrow,1)$, M.~Busaniche and R.~Cignoli (\cite{BC}) as well as C.~Tsinakis and A.~M.~Wille (\cite{TW}) introduced binary operations $\odot$ and $\Rightarrow$ on the full twist-product $(L^2,\sqcup,\sqcap)$ to be converted into a residuated lattice $(L^2,\sqcup,\sqcap,\odot,\Rightarrow,(1,1))$. It is known that if $\mathbf L=(L,\vee,\wedge,{}')$ is a distributive lattice with an antitone involution, $a\in L$ and $P_a(\mathbf L):=\{(x,y)\in L^2\mid x\wedge y\leq a\leq x\vee y\}$ then $(P_a(\mathbf L),\vee,\wedge,{}')$ is a Kleene lattice. If $\mathbf L$ is not distributive then the situation is different.

Our aim is to combine both of these approaches and hence ask for several questions as follows:
\begin{itemize}
\item When is $(P_a(\mathbf L),\sqcup,\sqcap)$ a sublattice of the full twist-product $(L^2,\sqcup,\sqcap)$, also in the case of a non-distributive lattice $\mathbf L$?
\item When is $P_a(\mathbf L)$ closed under operations $\odot$ and $\Rightarrow$ mentioned above?
\item When can $P_a(\mathbf L)$ be equipped with these operations forming an adjoint pair?
\item Can we define the operations $\odot$ and $\Rightarrow$ in a way different from that of \cite{BC} or \cite{TW} to obtain an integral residuated lattice on the full twist-product $(L^2,\sqcup,\sqcap)$?
\end{itemize}
We answer these question in our paper by giving sufficient and, in some cases, also necessary conditions under which we get a positive solution. Moreover, we present examples showing how our constructions work.

\section{Preliminaries}

We recall several concepts which will be used throughout the paper. Moreover, we recall some results already published on which our present study is based.

Let $\mathbf P=(P,\leq)$ be a poset. An {\em antitone involution} on $\mathbf P$ is a unary operation $'$ on $P$ satisfying
\begin{enumerate}[(i)]
\item $x\leq y$ implies $y'\leq x'$,
\item $x''\approx x$
\end{enumerate}
for all $x,y\in P$. A distributive lattice having an antitone involution is called a {\em De Morgan lattice} or a {\em De Morgan algebra}.

\begin{definition}\label{def1}
A {\em commutative residuated lattice} is an algebra $(L,\vee,\wedge,\cdot,\rightarrow,1)$ of type $(2,2,2,2,0)$ such that
\begin{enumerate}[{\rm(i)}]
\item $(L,\vee,\wedge)$ is a lattice,
\item $(L,\cdot,1)$ is a commutative monoid,
\item for all $x,y,z\in L$, $x\cdot y\leq z$ is equivalent to $x\leq y\rightarrow z$ {\rm(}{\em adjointness property}{\rm)}.
\end{enumerate}
$(L,\vee,\wedge,\cdot,\rightarrow,1)$ is called {\em integral} if $1$ is the top element of the lattice $(L,\vee,\wedge)$. A {\em commutative residuated lattice with $0$} is an algebra $(L,\vee,\wedge,\cdot,\rightarrow,0,1)$ of type $(2,2,2,2,0,0)$ such that $(L,\vee,\wedge,\cdot,\rightarrow,1)$ is a commutative residuated lattice and $0$ is the bottom element of $(L,\vee,\wedge)$. Let $(L,\vee,\wedge,\cdot,\rightarrow,0,1)$ be a commutative residuated lattice with $0$. Define $x':=x\rightarrow0$ for all $x\in L$. $(L,\vee,\wedge,\cdot,\rightarrow,0,1)$ is
\begin{itemize}
\item called a {\em bounded commutative residuated lattice} if $1$ is the top element of $(L,\vee,\wedge)$,
\item said to satisfy the {\em double negation law} if it satisfies the identity $x''\approx x$, i.e. \\
$(x\rightarrow0)\rightarrow0\approx x$.
\end{itemize}
\end{definition}

We say that the operations $\cdot$ and $\rightarrow$ form an {\em adjoint pair} if they satisfy the adjointness (iii) of Definition~\ref{def1}.

The following properties of integral commutative residuated lattices are well-known (cf.\ e.g.\ \cite B).

\begin{proposition}\label{prop1}
Let $(L,\vee,\wedge,\cdot,\rightarrow,1)$ be an integral commutative residuated lattice. \\
Then the following hold for all $x,y,z\in L$:
\begin{enumerate}[{\rm(i)}]
\item $x\leq y$ implies $x\cdot z\leq y\cdot z$,
\item $x\cdot y\leq x,y$,
\item $1\rightarrow x\approx x$,
\item $x\leq y\rightarrow x$,
\item $x\rightarrow y=1$ if and only if $x\leq y$,
\item $x\leq y$ implies $y\rightarrow z\leq x\rightarrow z$,
\item $x\leq y$ implies $z\rightarrow x\leq z\rightarrow y$,
\item $x\rightarrow(y\wedge z)\approx(x\rightarrow y)\wedge(x\rightarrow z)$,
\item $(x\cdot y)\rightarrow z\approx x\rightarrow(y\rightarrow z)$.
\end{enumerate}
\end{proposition}

Let $\mathbf L=(L,\vee,\wedge)$ be a lattice. By the {\em full twist-product} of $\mathbf L$ is meant the lattice $(L^2,\sqcup,\sqcap)$ where $\sqcup$ and $\sqcap$ are defined as follows:
\begin{align*}
(x,y)\sqcup(z,v) & :=(x\vee z,y\wedge v), \\
(x,y)\sqcap(z,v) & :=(x\wedge z,y\vee v)
\end{align*}
for all $(x,y),(z,v)\in L^2$. Hence $(x,y)\leq(z,v)$ if and only if both $x\leq z$ and $v\leq y$. Assume now that $(L,\vee,\wedge,\cdot,\rightarrow,1)$ is an integral commutative residuated lattice. In Theorem~3.1 in \cite{BC} which is a particular case of Corollary~3.6 in \cite{TW}, Busaniche and Cignoli introduced two additional binary operations $\odot$ and $\Rightarrow$ on its full twist-product $(L^2,\sqcup,\sqcap)$ as follows:
\begin{eqnarray}
& & (x,y)\odot(z,v):=(x\cdot z,(x\rightarrow v)\wedge(z\rightarrow y)),\label{equ1} \\
& & (x,y)\Rightarrow(z,v):=((x\rightarrow z)\wedge(v\rightarrow y),x\cdot v)\label{equ2}
\end{eqnarray}
for all $(x,y),(z,v)\in L^2$. They showed that $(L^2,\sqcup,\sqcap,\odot,\Rightarrow,(1,1))$ is again a commutative residuated lattice, i.e.\ $\odot$ and $\Rightarrow$ form an adjoint pair. For the convenience of the reader we provide a proof since it is not explicitly contained in \cite{BC} and \cite{TW}.

\begin{theorem}\label{th5}
Let $\mathbf L=(L,\vee,\wedge,\cdot,\rightarrow,1)$ be an integral commutative residuated lattice and $\odot$ and $\Rightarrow$ defined by {\rm(\ref{equ1})} and {\rm(\ref{equ2})}, respectively. Then $(L^2,\sqcup,\sqcap,\odot,\Rightarrow,(1,1))$ is a commutative residuated lattice.
\end{theorem}

\begin{proof}
Let $a,b,c,d,e,f\in L$.
\begin{enumerate}[(i)]
\item It is easy to see that $(L^2,\sqcup,\sqcap)$ is a lattice.
\item We prove that $(L^2,\odot,(1,1))$ is a commutative monoid. Because of (iii), (v), (viii) and (ix) of Proposition~\ref{prop1} we have
\begin{align*}
            (x,y)\odot(z,v) & \approx(x\cdot z,(x\rightarrow v)\wedge(z\rightarrow y))\approx(z\cdot x,(z\rightarrow y)\wedge(x\rightarrow v))\approx \\
                            & \approx(z,v)\odot(x,y), \\
((x,y)\odot(z,v))\odot(t,w) & \approx(x\cdot z,(x\rightarrow v)\wedge(z\rightarrow y))\odot(t,w)\approx \\
                            & \approx((x\cdot z)\cdot t,((x\cdot z)\rightarrow w)\wedge(t\rightarrow((x\rightarrow v)\wedge(z\rightarrow y))))\approx \\
                            & \approx(x\cdot(z\cdot t),((x\cdot z)\rightarrow w)\wedge(t\rightarrow(x\rightarrow v))\wedge(t\rightarrow(z\rightarrow y)))\approx \\								   
														& \approx(x\cdot(z\cdot t),(x\rightarrow(z\rightarrow w))\wedge(x\rightarrow(t\rightarrow v))\wedge((z\cdot t)\rightarrow y))\approx \\
                            & \approx(x\cdot(z\cdot t),(a\rightarrow((z\rightarrow w)\wedge(t\rightarrow v)))\wedge((z\cdot t)\rightarrow y))\approx \\
                            & \approx(x,y)\odot(z\cdot t,(z\rightarrow w)\wedge(t\rightarrow v)\approx \\
														& \approx(x,y)\odot((z,v)\odot(t,w)), \\
            (x,y)\odot(1,1) & \approx(x\cdot1,(x\rightarrow1)\wedge(1\rightarrow y))\approx(x,1\wedge y)\approx(x,y).
\end{align*}
\item Now we prove the adjointness property. The following are equivalent:
\begin{align*}
& (a,b)\odot(c,d)\leq(e,f), \\
& (a\cdot c,(a\rightarrow d)\wedge(c\rightarrow b))\leq(e,f), \\
& a\cdot c\leq e\text{ and }f\leq(a\rightarrow d)\wedge(c\rightarrow b), \\
& a\cdot c\leq e, f\leq a\rightarrow d\text{ and }f\leq c\rightarrow b, \\
& a\leq c\rightarrow e, a\leq f\rightarrow d\text{ and }c\cdot f\leq b, \\
& a\leq(c\rightarrow e)\wedge(f\rightarrow d)\text{ and }c\cdot f\leq b, \\
& (a,b)\leq((c\rightarrow e)\wedge(f\rightarrow d),c\cdot f), \\
& (a,b)\leq(c,d)\Rightarrow(e,f).
\end{align*}   
\end{enumerate}
\end{proof}

It is worth noticing that the operations $\odot$ and $\Rightarrow$ defined above are not independent. Namely one can be expressed by the other by using the antitone involution $'$ defined by $(x,y)':=(y,x)$. Namely,
\begin{align*}
       (x,y)\odot(z,v) & \approx(x\cdot z,(x\rightarrow v)\wedge(z\rightarrow y)\approx((x\rightarrow v)\wedge(z\rightarrow y),x\cdot z)'\approx \\
                       & \approx((x,y)\Rightarrow(v,z))'\approx((x,y)\Rightarrow(z,v)')', \\
(x,y)\Rightarrow (z,v) & \approx((x,y)\Rightarrow(z,v)'')''\approx((x,y)\odot(z,v)')'.
\end{align*}

Moreover, note that the residuated lattice $(L^2,\sqcup,\sqcap,\odot,\Rightarrow,(1,1))$ as defined above is not integral since the top element $(1,0)$ of the full twist-product is different from the neutral element $(1,1)$ of the monoid $(L^2,\odot,(1,1))$.

The following concept was introduced in \cite{Ch} .

A {\em pseudo-Kleene lattice} is an algebra $(L,\vee,\wedge,{}')$ of type $(2,2,1)$ such that the following hold for all $x,y\in L$:
\begin{enumerate}[(i)]
\item $\mathbf L=(L,\vee,\wedge)$ is a lattice,
\item $'$ is an antitone involution on $(L,\leq)$,
\item $x\wedge x'\leq y\vee y'$.
\end{enumerate}
(Here and in the rest of the paper $\leq$ denotes the induced order of the lattice $\mathbf L$.) If, moreover, $\mathbf L$ is distributive then $(L,\vee,\wedge,{}')$ is called a {\em Kleene lattice}.

\section{A construction of pseudo-Kleene lattices in the full twist-product}

Let $\mathbf L=(L,\vee,\wedge)$ be a lattice and $(L^2,\sqcup,\sqcap)$ its full twist-product. It is easy to see that $(L^2,\sqcup,\sqcap)$ is distributive if and only if so is $\mathbf L$. The following construction was introduced for distributive lattices in \cite{Ci} and generalized for posets by the authors in \cite{CL}: Let $a\in L$ and consider the following subset of $L^2$: 
\[
P_a(\mathbf L):=\{(x,y)\in L^2\mid x\wedge y\leq a\leq x\vee y\}.
\]
Since our paper \cite{CL} is devoted to posets and not to lattices, we are going to show that if $(P_a(\mathbf L),\sqcup,\sqcap)$ is a sublattice of $(L^2,\sqcup,\sqcap)$ then $(P_a(\mathbf L),\sqcup,\sqcap,{}')$ where the unary operation $'$ on $P_a(\mathbf L)$ is defined by $(x,y)':=(y,x)$ for all $(x,y)\in P_a(\mathbf L)$ is a pseudo-Kleene lattice.

\begin{theorem}\label{th2}
Let $\mathbf L=(L,\vee,\wedge)$ be a lattice and $a\in L$, assume that $(P_a(\mathbf L),\sqcup,\sqcap)$ is a sublattice of $(L^2,\sqcup,\sqcap)$ and put $(x,y)':=(y,x)$ for all $(x,y)\in L^2$. Then
\begin{enumerate}[{\rm(i)}]
\item $(P_a(\mathbf L),\sqcup,\sqcap,{}')$ is a pseudo-Kleene lattice,
\item the mapping $x\mapsto(x,a)$ is an embedding of $\mathbf L$ into $(P_a(\mathbf L),\sqcup,\sqcap)$,
\item $(P_a(\mathbf L),\sqcup,\sqcap)$ is distributive if and only if so is $\mathbf L$.
\end{enumerate}
\end{theorem}

\begin{proof}
Let $(b,c),(d,e)\in P_a(\mathbf L)$ and $f,g\in L$.
\begin{enumerate}[(i)]
\item The following are equivalent:
\begin{align*}
& (b,c)\leq(d,e), \\
& b\leq d\text{ and }e\leq c, \\
& e\leq c\text{ and }b\leq d, \\
& (e,d)\leq(c,b), \\
& (d,e)'\leq(b,c)'.
\end{align*}
Further, we have $(b,c)''=(c,b)'=(b,c)$. Thus $'$ is an antitone involution on $(P_a(\mathbf L),\sqcup,\sqcap)$. Moreover,
\begin{align*}
(b,c)\sqcap(b,c)' & =(b,c)\sqcap(c,b)=(b\wedge c,c\vee b)\leq(a,a)\leq(d\vee e,e\wedge d)= \\
                  & =(d,e)\sqcup(e,d)=(d,e)\sqcup(d,e)'
\end{align*}
proving that $(P_a(\mathbf L),\sqcup,\sqcap,{}')$ is a pseudo-Kleene lattice.
\item Since we have $(f,a)\leq(g,a)$ if and only if $f\leq g$, it is evident.
\item This can be easily checked.
\end{enumerate}
\end{proof}

In general, $(P_a(\mathbf L),\sqcup,\sqcap)$ need not be a sublattice of $(L^2,\sqcup,\sqcap)$.

\begin{example}
Consider the lattice $\mathbf N_5=(N_5,\vee,\wedge)$ depicted in Figure~1:

\vspace*{-4mm}

\begin{center}
\setlength{\unitlength}{7mm}
\begin{picture}(4,8)
\put(2,1){\circle*{.3}}
\put(1,3){\circle*{.3}}
\put(3,4){\circle*{.3}}
\put(1,5){\circle*{.3}}
\put(2,7){\circle*{.3}}
\put(2,1){\line(-1,2)1}
\put(2,1){\line(1,3)1}
\put(1,3){\line(0,1)2}
\put(2,7){\line(-1,-2)1}
\put(2,7){\line(1,-3)1}
\put(1.85,.25){$0$}
\put(.3,2.85){$a$}
\put(.3,4.85){$b$}
\put(3.4,3.85){$c$}
\put(1.85,7.4){$1$}
\put(1.2,-.75){{\rm Fig.\ 1}}
\end{picture}
\end{center}

\vspace*{4mm}

Then $(a,1),(c,b)\in P_a(\mathbf N_5)$, but $(a,1)\sqcup(c,b)=(a\vee c,1\wedge b)=(1,b)\notin P_a(\mathbf N_5)$ since $1\wedge b=b\not\leq a$. This shows that $P_a(\mathbf N_5)$ is not a sublattice of the full twist-product $(N_5^2,\sqcup,\sqcap)$ of $\mathbf N_5$.
\end{example}

We can give a necessary and sufficient condition for $(P_a(\mathbf L),\sqcup,\sqcap)$ being a sublattice of $(L^2,\sqcup,\sqcap)$.

\begin{theorem}\label{th3}
Let $\mathbf L=(L,\vee,\wedge)$ be a lattice and $a\in L$. Then $(P_a(\mathbf L),\sqcup,\sqcap)$ is a sublattice of $(L^2,\sqcup,\sqcap)$ if and only if the following condition holds for all $x,y,z,v\in L$:
\begin{align*}
& (x\wedge y)\vee(z\wedge v)\leq a\leq(x\vee y)\wedge(z\vee v)\text{ implies} \\
& ((x\vee z)\wedge y\wedge v)\vee(x\wedge z\wedge(y\vee v))\leq a\leq (x\vee z\vee(y\wedge v))\wedge((x\wedge z)\vee y\vee v).
\end{align*}
\end{theorem}

\begin{proof}
Let $b,c,d,e\in L$.
\item The following are equivalent:
\begin{align*}
& (b,c),(d,e)\in P_a(\mathbf L), \\
& b\wedge c\leq a\leq b\vee c\text{ and }d\wedge e\leq a\leq d\vee e, \\
& (b\wedge c)\vee(d\wedge e)\leq a\leq(b\vee c)\wedge(d\vee e).
\end{align*}
Moreover, the following are equivalent:
\begin{align*}
& (b,c)\sqcup(d,e)\in P_a(\mathbf L), \\
& (b\vee d,c\wedge e)\in P_a(\mathbf L), \\
& (b\vee d)\wedge c\wedge e\leq a\leq b\vee d\vee(c\wedge e).
\end{align*}
Finally, the following are equivalent:
\begin{align*}
& (b,c)\sqcap(d,e)\in P_a(\mathbf L), \\
& (b\wedge d,c\vee e)\in P_a(\mathbf L), \\
& b\wedge d\wedge(c\vee e)\leq a\leq(b\wedge d)\vee c\vee e.
\end{align*}
\end{proof} 

\begin{corollary}\label{cor1}
Let $\mathbf L=(L,\vee,\wedge)$ be a distributive lattice and $a\in L$. Then $(P_a(\mathbf L),\sqcup,$ $\sqcap)$ is a sublattice of $(L^2,\sqcup,\sqcap)$ and $(P_a(\mathbf L),\sqcup,\sqcap,{}')$ where the antitone involution is given by $(x,y)':=(y,x)$ for all $(x,y)\in P_a(\mathbf L)$ is a Kleene lattice.
\end{corollary}

\begin{proof}
If $b,c,d,e\in L$ and
\[
(b\wedge c)\vee(d\wedge e)\leq a\leq(b\vee c)\wedge(d\vee e)
\]
then
\begin{align*}
((b\vee d)\wedge c\wedge e)\vee(b\wedge d\wedge(c\vee e)) & =(b\wedge c\wedge e)\vee(d\wedge c\wedge e)\vee(b\wedge d\wedge c)\vee(b\wedge d\wedge e)\leq \\
                                                          & \leq(b\wedge c)\vee(d\wedge e)\vee(b\wedge c)\vee(d\wedge e)\leq a\leq \\
                                                          & \leq(b\vee c)\wedge(d\vee e)\wedge(b\vee c)\wedge(d\vee e)\leq \\
                                                          & \leq(b\vee d\vee c)\wedge(b\vee d\vee e)\wedge(b\vee c\vee e)\wedge(d\vee c\vee e)\leq \\
                                                          & \leq(b\vee d\vee(c\wedge e))\wedge((b\wedge d)\vee c\vee e).
\end{align*}
The rest of proof follows by Theorem~\ref{th2}.
\end{proof}

The following example shows a distributive lattice $\mathbf L$ having an element $a$ such that $(P_a(\mathbf L),\sqcup,\sqcap)$ is a sublattice of the full twist-product $(L^2,\sqcup,\sqcap)$.

\begin{example}\label{ex1}
Consider the lattice $\mathbf L=(L,\vee,\wedge)$ visualized in Figure~2:

\vspace*{-2mm}

\begin{center}
\setlength{\unitlength}{7mm}
\begin{picture}(1,8)
\put(0,1){\circle*{.3}}
\put(0,3){\circle*{.3}}
\put(0,5){\circle*{.3}}
\put(0,7){\circle*{.3}}
\put(0,1){\line(0,1)6}
\put(-.15,.25){$0$}
\put(.4,2.85){$a$}
\put(.4,4.85){$b$}
\put(-.15,7.4){$1$}
\put(-.8,-.75){{\rm Fig.\ 2}}
\end{picture}
\end{center}

\vspace*{4mm}

If one defines binary operations $\cdot$ and $\rightarrow$ on $L$ by
\[
x\cdot y:=x\wedge y\text{ and }x\rightarrow y:=\left\{
\begin{array}{ll}
1 & \text{if }x\leq y, \\
y & \text{otherwise},
\end{array}
\right.
\]
then $(L,\vee,\wedge,\cdot,\rightarrow,1)$ is an distributive integral commutative residuated lattice. With respect to the binary operations $\odot$ and $\Rightarrow$ defined by {\rm(\ref{equ1})} and {\rm(\ref{equ2})}, respectively, $(P_a(\mathbf L),\sqcup,\sqcap,$ $\odot,\Rightarrow,(0,1),$ $(1,0))$ is a bounded commutative residuated lattice. According to Corollary~\ref{cor1}, $(P_a(\mathbf L),\sqcup,\sqcap)$ is a sublattice if $(L^2,\sqcup,\sqcap)$. The Hasse diagram of $(P_a(\mathbf L),\sqcup,\sqcap)$ is depicted in Figure~3.

\vspace*{-2mm}

\begin{center}
\setlength{\unitlength}{7mm}
\begin{picture}(6,14)
\put(5,1){\circle*{.3}}
\put(3,3){\circle*{.3}}
\put(7,3){\circle*{.3}}
\put(1,5){\circle*{.3}}
\put(5,5){\circle*{.3}}
\put(3,7){\circle*{.3}}
\put(1,9){\circle*{.3}}
\put(5,9){\circle*{.3}}
\put(-1,11){\circle*{.3}}
\put(3,11){\circle*{.3}}
\put(1,13){\circle*{.3}}
\put(5,1){\line(-1,1)2}
\put(5,1){\line(1,1)2}
\put(3,3){\line(-1,1)2}
\put(3,3){\line(1,1)2}
\put(7,3){\line(-1,1)2}
\put(1,5){\line(1,1)4}
\put(5,5){\line(-1,1)4}
\put(1,9){\line(-1,1)2}
\put(1,9){\line(1,1)2}
\put(5,9){\line(-1,1)2}
\put(-1,11){\line(1,1)2}
\put(3,11){\line(-1,1)2}
\put(4.35,.25){$(0,1)$}
\put(1.4,2.85){$(0,b)$}
\put(7.3,2.85){$(a,1)$}
\put(-.6,4.85){$(0,a)$}
\put(5.3,4.85){$(a,b)$}
\put(3.3,6.85){$(a,a)$}
\put(-.6,8.85){$(b,a)$}
\put(5.3,8.85){$(a,0)$}
\put(-2.6,10.85){$(1,a)$}
\put(3.3,10.85){$(b,0)$}
\put(.35,13.45){$(1,0)$}
\put(2.2,-.75){{\rm Fig.\ 3}}
\end{picture} 
\end{center}
\end{example}

\vspace{4mm}

In the following, a special role will play the lattices $P_a(\mathbf L)$ all elements of which are comparable with $(a,a)$. We can characterize them as follows.

\begin{theorem}\label{th1}
Let $\mathbf L=(L,\vee,\wedge)$ be a lattice and $a\in L$. Then the following are equivalent:
\begin{enumerate}[{\rm(i)}]
\item $P_a(\mathbf L)\subseteq\{(x,y)\in L^2\mid(x,y)\text{ is comparable with }(a,a)\}$
\item $P_a(\mathbf L)=\{(x,y)\in L^2\mid(x,y)\text{ is comparable with }(a,a)\}$
\item Every element of $L$ is comparable with $a$ and $a$ is $\vee$-irreducible and $\wedge$-irreducible.
\end{enumerate}
If this is the case then $(P_a(\mathbf L),\sqcup,\sqcap)$ is a sublattice of $(L^2,\sqcup,\sqcap)$.
\end{theorem}

\begin{proof}
Let $b,c\in L$. \\
(i) and (ii) are equivalent since $\{(x,y)\in L^2\mid(x,y)\text{ is comparable with }(a,a)\}\subseteq P_a(\mathbf L)$. \\
(i) $\Rightarrow$ (iii): \\
Since $(a,a)\leq(a,b)$ or $(a,b)\leq(a,a)$ we have $b\leq a$ or $a\leq b$. If $a=b\vee c$ then $b,c<a$ would imply $(b,c)\in P_a(\mathbf L)$ and $(b,c)\parallel(a,a)$, a contradiction. Hence $a$ is $\vee$-irreducible. If $a=b\wedge c$ then $a<b,c$ would imply $(b,c)\in P_a(\mathbf L)$ and $(b,c)\parallel(a,a)$, a contradiction. Hence $a$ is $\wedge$-irreducible. \\
(iii) $\Rightarrow$ (i): \\
Let $(b,c)\in P_a(\mathbf L)$. Then $b\wedge c\leq a\leq b\vee c$. \\
If $b=a$ then $(b,c)=(a,c)$ is comparable with $(a,a)$. \\
If $c=a$ then $(b,c)=(b,a)$ is comparable with $(a,a)$. \\
$b,c<a$ is impossible because of $a\leq b\vee c$. \\
$a<b,c$ is impossible because of $b\wedge c\leq a$. \\
If $b<a<c$ then $(b,c)\leq(a,a)$. \\
If $c<a<b$ then $(a,a)\leq(b,c)$. \\
Now assume that (ii) holds and let $(b,c),(d,e)\in P_a(\mathbf L)$. \\
If $(b,c),(d,e)\leq(a,a)$ then
\begin{align*}
(b,c)\sqcup(d,e) & =(b\vee d,c\wedge e)\leq(a,a), \\
(b,c)\sqcap(d,e) & =(b\wedge d,c\vee e)\leq(a,a).
\end{align*}
If $(b,c)\leq(a,a)\leq(d,e)$ then
\begin{align*}
(b,c)\sqcup(d,e) & =(b\vee d,c\wedge e)\geq(a,a), \\
(b,c)\sqcap(d,e) & =(b\wedge d,c\vee e)\leq(a,a),
\end{align*}
If $(d,e)\leq(a,a)\leq(b,c)$ then
\begin{align*}
(b,c)\sqcup(d,e) & =(b\vee d,c\wedge e)\geq(a,a), \\
(b,c)\sqcap(d,e) & =(b\wedge d,c\vee e)\leq(a,a).
\end{align*}
If $(a,a)\leq(b,c),(d,e)$ then
\begin{align*}
(b,c)\sqcup(d,e) & =(b\vee d,c\wedge e)\geq(a,a), \\
(b,c)\sqcap(d,e) & =(b\wedge d,c\vee e)\geq(a,a).
\end{align*}
Hence $(P_a(\mathbf L),\sqcup,\sqcap)$ is a sublattice of $(L^2,\sqcup,\sqcap)$.
\end{proof}

\begin{example}
We can see that the lattice $\mathbf L$ and its element $a$ from Example~\ref{ex1} satisfy the conditions of Theorem~\ref{th1} {\rm(iii)}, hence all elements of $P_a(\mathbf L)$ are comparable with the element $(a,a)$, see Figure~3.
\end{example}

\section{Adjoint pairs in $P_a(\mathbf L)$}

Since the element $(1,1)$ of $L^2$ does not belong to $P_a(\mathbf L)$ unless $a=1$, we cannot expect that $(P_a(\mathbf L),\sqcup,\sqcap)$ will be a residuated lattice with respect to operations $\odot$ and $\Rightarrow$ defined by (\ref{equ1}) and (\ref{equ2}), respectively. On the other hand, it would be important to know when $P_a(\mathbf L)$ is closed with respect to $\odot$ and $\Rightarrow$ because then they form an adjoint pair. Hence, if the pseudo-Kleene lattice $P_a(\mathbf L)$ represents a certain logic where $\odot$ is conjunction and $\Rightarrow$ is implication then from the trivial inequality
\[
x\Rightarrow y\leq x\Rightarrow y
\]
we infer by adjointness
\[
(x\Rightarrow y)\odot x\leq y,
\]
in other words, the propositional value of $y$ is at least as high as the propositional values of the conjunction of $x\Rightarrow y$ and $x$. This means that this logic satisfies {\em Modus Ponens} in the fuzzy modification and hence this pseudo-Kleene logic enables deduction.

Now we are ready to state and prove one of our main results.

\begin{theorem}\label{th4}
Let $(L,\vee,\wedge,\cdot,\rightarrow,1)$ be an integral commutative residuated lattice and $a$ an idempotent {\rm(}with respect to $\cdot${\rm)} $\vee$-irreducible and $\wedge$-irreducible element of $L$ which is comparable with every element of $L$  and put $\mathbf L:=(L,\vee,\wedge)$. Then $(P_a(\mathbf L),\sqcup,\sqcap,\odot,\Rightarrow)$ is a subalgebra of $(L^2,\sqcup,\sqcap,\odot,\Rightarrow)$ and hence $\odot$ and $\Rightarrow$ form an adjoint pair if and only if the following two conditions hold for all $x,y\in L$:
\begin{eqnarray}
& & a\cdot x<a\text{ implies }a\cdot x=0,\label{equ3} \\
& & a<x\cdot y\text{ implies }(x\rightarrow a)\wedge(y\rightarrow a)=a.\label{equ4}
\end{eqnarray}
\end{theorem}

\begin{proof}
Let $(b,c),(d,e)\in P_a(\mathbf L)$. According to Theorem~\ref{th1},
\[
P_a(\mathbf L)=\{(x,y)\in L^2\mid (x,y)\text{ is comparable with }(a,a)\}
\]
and we have that $(P_a(\mathbf L),\sqcup,\sqcap)$ is a sublattice of $(L^2,\sqcup,\sqcap)$. Since $P_a(\mathbf L)$ is closed with respect to $'$, it is closed with respect to $\Rightarrow$ if it is closed with respect to $\odot$. Hence, we need only to check when $P_a(\mathbf L)$ is closed with respect to $\odot$.
\begin{enumerate}[(i)]
\item Assume $(b,c),(d,e)\leq(a,a)$. \\
Because of (ii) and (iv) of Proposition~\ref{prop1} we have
\[
(b,c)\odot(d,e)=(b\cdot d,(b\rightarrow e)\wedge(d\rightarrow c))\leq(a,a).
\]
\item Assume $(b,c)\leq(a,a)\leq(d,e)$. \\
Because of (ii) of Proposition~\ref{prop1} we have $b\cdot d\leq a$. \\
If $b\cdot d=a$ then $(b,c)\odot(d,e)=(b\cdot d,(b\rightarrow e)\wedge(d\rightarrow c))$ is comparable with $(a,a)$. \\
If $b\cdot d<a$ then $(b,c)\odot(d,e)=(b\cdot d,(b\rightarrow e)\wedge(d\rightarrow c))$ is comparable with $(a,a)$ \\
\hspace*{1cm} if and only if $a\leq b\rightarrow e$.
\item Assume $(d,e)\leq(a,a)\leq(b,c)$. \\
Because of the commutativity of $\odot$ this case reduces to the previous one.
\item Assume $(a,a)\leq(b,c),(d,e)$. \\
Because of (i) of Proposition~\ref{prop1} we have $a\leq b\cdot d$. \\
If $a=b\cdot d$ then $(b,c)\odot(d,e)=(b\cdot d,(b\rightarrow e)\wedge(d\rightarrow c))$ is comparable with $(a,a)$. \\
If $a<b\cdot d$ then $(b,c)\odot(d,e)=(b\cdot d,(b\rightarrow e)\wedge(d\rightarrow c))$ is comparable with $(a,a)$ \\
\hspace*{1cm} if and only if $(b\rightarrow e)\wedge(d\rightarrow c)\leq a$.
\end{enumerate}
Hence $(P_a(\mathbf L),\sqcup,\sqcap,\odot,\Rightarrow)$ is a subalgebra of $(L^2,\sqcup,\sqcap,\odot,\Rightarrow)$ if and only if the following statements hold:
\begin{enumerate}[(a)]
\item $b,e\leq a\leq c,d$ and $b\cdot d<a$ imply $a\leq b\rightarrow e$.
\item $c,e\leq a\leq b,d$ and $a<b\cdot d$ imply $(b\rightarrow e)\wedge(d\rightarrow c)\leq a$.
\end{enumerate}
Because of (i) and (vii) of Proposition~\ref{prop1}, (a) is equivalent to the following statements:
\begin{align*}
& b\cdot a<a\text{ implies }a\leq b\rightarrow0, \\
& a\cdot b<a\text{ implies }a\cdot b\leq0, \\
& a\cdot b<a\text{ implies }a\cdot b=0, \\
& (\ref{equ3}).
\end{align*}
Moreover, because of (iv) and (vii) of Proposition~\ref{prop1}, (b) is equivalent to the following statements:
\begin{align*}
& a<b\cdot d\text{ implies }(b\rightarrow a)\wedge(d\rightarrow a)\leq a, \\
& a<b\cdot d\text{ implies }(b\rightarrow a)\wedge(d\rightarrow a)=a, \\
& (\ref{equ4}).
\end{align*}
\end{proof}

\begin{corollary}
Let $\mathbf L=(L,\vee,\wedge,\cdot,\rightarrow,1)$ be an integral commutative distributive residuated lattice and $a\in L$ with $a\cdot a=a$ and assume that every element of $P_a(\mathbf L)$ is comparable with $(a,a)$. Then $(P_a(\mathbf L),\sqcup,\sqcap,{}')$ where $(x,y)':=(y,x)$ for all $(x,y)\in P_a(\mathbf L)$ is a Kleene lattice and $\odot$ and $\Rightarrow$ form an adjoint pair if and only if {\rm(\ref{equ3})} and {\rm(\ref{equ4})} hold.
\end{corollary}

\begin{example}
Consider the lattice $\mathbf L=(L,\vee,\wedge)$ with element $a$ from Example~\ref{ex1}. One can easily check that $\mathbf L$ satisfies the conditions of Theorem~\ref{th4} and hence $(P_a(\mathbf L),\sqcup,\sqcap,\odot,$ $\Rightarrow)$ is a subalgebra of $(L^2,\sqcup,\sqcap,\odot,\Rightarrow)$.
\end{example}

\begin{lemma}\label{lem1}
Let $\mathbf L=(L,\vee,\wedge,\cdot,\rightarrow,1)$ be a distributive commutative residuated lattice and $a\in L$. Then $(P_a(\mathbf L),\sqcup,\sqcap)$ is a distributive sublattice of the full twist-product $(L^2,\sqcup,\sqcap)$ closed with respect to $\odot$ {\rm(}and hence also with respect to $\Rightarrow${\rm)} if and only if for all $(b,c),(d,e)\in P_a(\mathbf L)$
\[
(b\cdot d)\wedge(b\rightarrow e)\wedge(d\rightarrow c)\leq a\leq(b\cdot d)\vee(b\rightarrow e)\text{ and }a\leq(b\cdot d)\vee(d\rightarrow c).
\]
\end{lemma}

\begin{proof}
According to Theorem~\ref{th2} and Corollary~\ref{cor1}, $(P_a(\mathbf L),\sqcup,\sqcap)$ is a distributive sublattice of $(L^2,\sqcup,\sqcap)$. Let $(b,c),(d,e)\in P_a(\mathbf L)$ and put $f:=b\cdot d$, $g:=b\rightarrow e$, $h:=d\rightarrow c$ and $i:=g\wedge h$. Then the following are equivalent:
\begin{align*}
(b,c)\odot(d,e) & \in P_a(\mathbf L), \\
          (f,i) & \in P_a(\mathbf L), \\
      f\wedge i & \leq a\leq f\vee i, \\
      f\wedge i & \leq a\leq f\vee(g\wedge h), \\
      f\wedge i & \leq a\leq(f\vee g)\wedge(f\vee h), \\
      f\wedge i & \leq a\leq f\vee g\text{ and }a\leq f\vee h.
\end{align*}
\end{proof}

\begin{corollary}\label{cor2}
Let $\mathbf L=(L,\vee,\wedge,\cdot,\rightarrow,0,1)$ be a distributive bounded commutative residuated lattice and $a$ an atom of $\mathbf L$. Then $(P_a(\mathbf L),\sqcup,\sqcap)$ is a distributive sublattice of the full twist-product $(L^2,\sqcup,\sqcap)$ closed with respect to $\odot$ {\rm(}and hence also with respect to $\Rightarrow${\rm)} if and only if for all $(b,c),(d,e)\in P_a(\mathbf L)$ either {\rm(i)} or {\rm(ii)} hold:
\begin{enumerate}[{\rm(i)}]
\item $(b\cdot d)\wedge(b\rightarrow e)\wedge(d\rightarrow c)=a$,
\item $(b\cdot d)\wedge(b\rightarrow e)\wedge(d\rightarrow c)=0$ and $(a\leq b\cdot d$ or $a\leq(b\rightarrow e)\wedge(d\rightarrow c))$.
\end{enumerate}
\end{corollary}

\begin{proof}
Let $(b,c),(d,e)\in P_a(\mathbf L)$ and put $f:=b\cdot d$, $g:=b\rightarrow e$, $h:=d\rightarrow c$ and $i:=g\wedge h$. According to Lemma~\ref{lem1}, $(P_a(\mathbf L),\sqcup,\sqcap)$ is a distributive sublattice of the full twist-product $(L^2,\sqcup,\sqcap)$ and $(b,c)\odot(d,e)\in P_a(\mathbf L)$ is equivalent to $(f\wedge i\leq a\leq f\vee g$ and $a\leq f\vee h)$. Now $f\wedge i=a$ implies $a\leq f\vee i$. Using the fact that $a$ is an atom of $\mathbf L$ we see that the following are equivalent:
\begin{align*}
                         a & \leq f\vee g, \\
          a\wedge(f\vee g) & =a, \\
(a\wedge f)\vee(a\wedge g) & =a, \\
               a\wedge f=a & \text{ or }a\wedge g=a, \\
                   a\leq f & \text{ or }a\leq g.
\end{align*}
Analogously, $a\leq f\vee h$ is equivalent to $(a\leq f$ or $a\leq h)$. Finally, the following are equivalent:
\begin{align*}
(b,c)\odot(d,e) & \in P_a(\mathbf L), \\
    f\wedge i=a & \text{ or }(f\wedge i=0\text{ and }(a\leq f\text{ or }a\leq g)\text{ and }(a\leq f\text{ or }a\leq h)), \\
    f\wedge i=a & \text{ or }(f\wedge i=0\text{ and }(a\leq f\text{ or }(a\leq g)\text{ and }a\leq h))), \\
    f\wedge i=a & \text{ or }(f\wedge i=0\text{ and }(a\leq f\text{ or }a\leq g\wedge h)).
\end{align*}
\end{proof}

Analogously as in Corollary~\ref{cor2}, we can consider the operation $\Rightarrow$ instead of $\odot$ and prove a similar result.

\begin{lemma}
Let $\mathbf L=(L,\vee,\wedge,\cdot,\rightarrow,0,1)$ be a distributive bounded commutative residuated lattice and $a$ an atom of $\mathbf L$. Then $(P_a(\mathbf L),\sqcup,\sqcap)$ is a distributive sublattice of the full twist-product $(L^2,\sqcup,\sqcap)$ closed with respect to $\Rightarrow$ {\rm(}and hence also with respect to $\odot${\rm)} if and only if for all $(b,c),(d,e)\in P_a(\mathbf L)$ either {\rm(i)} or {\rm(ii)} hold:
\begin{enumerate}[{\rm(i)}]
\item $(b\rightarrow d)\wedge(e\rightarrow c)\wedge(b\cdot e)=a$,
\item $(b\rightarrow d)\wedge(e\rightarrow c)\wedge(b\cdot e)=0$ and $(a\leq(b\rightarrow d)\wedge(e\rightarrow c)$ or $a\leq b\cdot e)$.
\end{enumerate}
\end{lemma}

\begin{proof}
Let $(b,c),(d,e)\in P_a(\mathbf L)$ and put $f:=b\rightarrow d$, $g:=e\rightarrow c$, $h:=b\cdot e$ and $i:=f\wedge g$. According to Theorem~\ref{th2} and Corollary~\ref{cor1}, $(P_a(\mathbf L),\sqcup,\sqcap)$ is a distributive sublattice of the full twist-product $(L^2,\sqcup,\sqcap)$. Now the following are equivalent:
\begin{align*}
(b,c)\odot(d,e) & \in P_a(\mathbf L), \\
(i,h) & \in P_a(\mathbf L), \\
i\wedge h & \leq a\leq i\vee h, \\
i\wedge h & =a\text{ or }(i\wedge h=0\text{ and }a\leq i\vee h), \\
i\wedge h & =a\text{ or }(i\wedge h=0\text{ and }(a\leq i\text{ or }a\leq h)).
\end{align*}
(that $a\leq i\vee h$ is equivalent to $(a\leq i$ or $a\leq h)$ follows like in the proof of Corollary~\ref{cor2}).
\end{proof}

\begin{example}\label{ex2}
Consider the lattice $\mathbf L=(L,\vee,\wedge)$ shown in Figure~4:

\vspace*{-2mm}

\begin{center}
\setlength{\unitlength}{7mm}
\begin{picture}(8,8)
\put(3,1){\circle*{.3}}
\put(1,3){\circle*{.3}}
\put(5,3){\circle*{.3}}
\put(3,5){\circle*{.3}}
\put(7,5){\circle*{.3}}
\put(5,7){\circle*{.3}}
\put(3,1){\line(-1,1)2}
\put(3,1){\line(1,1)4}
\put(1,3){\line(1,1)4}
\put(5,3){\line(-1,1)2}
\put(7,5){\line(-1,1)2}
\put(2.875,.25){$0$}
\put(.3,2.85){$b$}
\put(5.4,2.85){$a$}
\put(2.3,4.85){$c$}
\put(7.4,4.85){$d$}
\put(4.875,7.4){$1$}
\put(3.2,-.75){{\rm Fig.\ 4}}
\end{picture}
\end{center}

\vspace*{4mm}

According to Corollary~\ref{cor1}, $(P_a(\mathbf L),\sqcup,\sqcap)$ is a sublattice of the full twist-product $(L^2,\sqcup,\sqcap)$. The Hasse diagram of $(P_a(\mathbf L),\sqcup,\sqcap)$ is depicted in Figure~5.

\vspace*{-2mm}

\begin{center}
\setlength{\unitlength}{7mm}
\begin{picture}(10,14)
\put(5,1){\circle*{.3}}
\put(3,3){\circle*{.3}}
\put(5,3){\circle*{.3}}
\put(7,3){\circle*{.3}}
\put(1,5){\circle*{.3}}
\put(3,5){\circle*{.3}}
\put(5,5){\circle*{.3}}
\put(7,5){\circle*{.3}}
\put(1,7){\circle*{.3}}
\put(3,7){\circle*{.3}}
\put(5,7){\circle*{.3}}
\put(7,7){\circle*{.3}}
\put(9,7){\circle*{.3}}
\put(3,9){\circle*{.3}}
\put(5,9){\circle*{.3}}
\put(7,9){\circle*{.3}}
\put(9,9){\circle*{.3}}
\put(3,11){\circle*{.3}}
\put(5,11){\circle*{.3}}
\put(7,11){\circle*{.3}}
\put(5,13){\circle*{.3}}
\put(5,1){\line(-1,1)4}
\put(5,1){\line(0,1)2}
\put(5,1){\line(1,1)2}
\put(3,3){\line(0,1)2}
\put(3,3){\line(1,1)2}
\put(5,3){\line(-1,1)4}
\put(5,3){\line(1,1)4}
\put(7,3){\line(-1,1)4}
\put(7,3){\line(0,1)4}
\put(1,5){\line(0,1)2}
\put(1,5){\line(1,1)2}
\put(3,5){\line(1,1)4}
\put(5,5){\line(0,1)4}
\put(7,5){\line(-1,1)4}
\put(1,7){\line(1,1)4}
\put(3,7){\line(0,1)4}
\put(7,7){\line(-1,1)4}
\put(7,7){\line(1,1)2}
\put(9,7){\line(-1,1)4}
\put(9,7){\line(0,1)2}
\put(5,9){\line(1,1)2}
\put(7,9){\line(0,1)2}
\put(9,9){\line(-1,1)4}
\put(3,11){\line(1,1)2}
\put(5,11){\line(0,1)2}
\put(4.35,.25){$(0,1)$}
\put(1.4,2.85){$(0,d)$}
\put(5.3,2.85){$(a,1)$}
\put(7.3,2.85){$(0,c)$}
\put(-.6,4.85){$(b,d)$}
\put(1.4,4.85){$(a,d)$}
\put(5.3,4.85){$(0,a)$}
\put(7.3,4.85){$(a,c)$}
\put(-.6,6.85){$(c,d)$}
\put(1.4,6.85){$(b,a)$}
\put(5.3,6.85){$(a,a)$}
\put(7.3,6.85){$(a,b)$}
\put(9.3,6.85){$(d,c)$}
\put(1.4,8.85){$(c,a)$}
\put(5.3,8.85){$(a,0)$}
\put(7.3,8.85){$(d,a)$}
\put(9.3,8.85){$(d,b)$}
\put(1.4,10.85){$(c,0)$}
\put(5.3,10.85){$(1,a)$}
\put(7.3,10.85){$(d,0)$}
\put(4.35,13.45){$(1,0)$}
\put(4.2,-.75){{\rm Fig.\ 5}}
\end{picture}
\end{center}

\vspace*{4mm}

Define an antitone involution $'$ on $(L,\leq)$ and binary operations $\cdot$ and $\rightarrow$ on $L$ by
\[
\begin{array}{c|cccccc}
x  & 0 & a & b & c & d & 1 \\
\hline
x' & 1 & c & d & a & b & 0
\end{array}
\]
and
\[
x\cdot y=\left\{
\begin{array}{ll}
x\wedge y\wedge b & \text{if }x,y\in\{a,c\}, \\
x\wedge y         & \text{otherwise}
\end{array}
\right.\quad x\rightarrow y=\left\{
\begin{array}{ll}
x'\vee y\vee d & \text{if }x,y\in\{a,c\}, \\
x'\vee y       & \text{otherwise}
\end{array}
\right.
\]
for all $x,y\in L$. Then $(L,\vee,\wedge,\cdot,\rightarrow,1)$ is an integral commutative residuated lattice and $x'=x\rightarrow0$ for all $x\in L$. Hence there holds the double negation law. Since $a$ is neither idempotent with respect to $\cdot$ nor meet-irreducible nor comparable with all elements of $L$, we cannot apply Theorem~\ref{th4}. However, since $\mathbf L$ is distributive, $a$ is atom of $\mathbf L$ and conditions {\rm(i)} and {\rm(ii)} of Corollary~\ref{cor2} are satisfied, $P_a(\mathbf L)$ is closed with respect to $\odot$ and hence also with respect to $\Rightarrow$.
\end{example}

If $\mathbf L$ denotes the lattice from Example~\ref{ex2} then $(c,d),(d,b)\in P_d(\mathbf L)$, but
\begin{align*}
(c,d)\odot(d,b) & =(c\cdot d,(c\rightarrow b)\wedge(d\rightarrow d))=(c\wedge d,(c'\vee b)\wedge(d'\vee d))= \\
                & =(a,(a\vee b)\wedge(b\vee d))=(a,c\wedge1)=(a,c)\notin P_d(\mathbf L)
\end{align*}
since $d\not\leq c=a\vee c$. This shows that $P_d(\mathbf L)$ is not closed with respect to $\odot$ (and hence also not with respect to $\Rightarrow$).

If $\mathbf L$ satisfies the double negation law then, because of (vi) of Proposition~\ref{prop1}, $'$ is an antitone involution on $(L,\leq)$. Two elements $a$ and $b$ of $L$ are said to be {\em orthogonal to each other} (shortly, $a\perp b$) if $a\leq b'$. If $\mathbf L$ satisfies the double negation law then this is equivalent to $b\leq a'$. $(L^2,\sqcup,\sqcap,\odot,\Rightarrow,(0,1),(1,1))$ is said to satisfy the {\em double negation law for orthogonal elements} if $(x,y)''=(x,y)$ for all $(x,y)\in L^2$ with $x\perp y$ where $(x,y)':=(x,y)\Rightarrow(0,1)$ for all $(x,y)\in L^2$.

\begin{theorem}
Let $\mathbf L=(L,\vee,\wedge,\cdot,\rightarrow,0,1)$ be a bounded commutative residuated lattice satisfying the double negation law and $a\in L$. Then the full twist-product $(L^2,\sqcup,\sqcap,\odot,\Rightarrow,$ $ (0,1),(1,1))$ is a commutative residuated lattice with zero-element $(0,1)$ satisfying the double negation law for orthogonal elements.
\end{theorem}

\begin{proof}
If $a,b\in L$ and $a\perp b$ then $b\leq a'$ and hence
\begin{align*}
 (a,b)' & =(a,b)\Rightarrow(0,1)=((a\rightarrow0)\wedge(1\rightarrow b),a\cdot1)=(a'\wedge b,a)=(b,a), \\
(a,b)'' & =(b,a)'=(a,b).
\end{align*}
\end{proof}

\section{An alternative construction of adjoint operations}

In this section we show that the operations $\odot$ and $\Rightarrow$ on the full twist-product $(L^2,\sqcup,\sqcap)$ can be defined also in a way different from (\ref{equ1}) and (\ref{equ2}) such that $(L^2,\sqcup,\sqcap,\odot,\Rightarrow,(0,1),$ $(1,0))$ becomes a bounded commutative residuated lattice. We formulate it as follows.

\begin{theorem}\label{th6}
Let $(L,\vee,\wedge,\cdot,\rightarrow,0,1)$ be a bounded commutative residuated lattice satisfying the double negation law and define $x':=x\rightarrow0$ for all $x\in L$ and
\begin{eqnarray}
& & (x,y)\odot(z,v):=(x\cdot z,(y'\cdot v')'),\label{equ5} \\
& & (x,y)\Rightarrow(z,v):=(x\rightarrow z,(y'\rightarrow v')')\label{equ6}
\end{eqnarray}
for all $(x,y),(z,v)\in L^2$. Then $(L^2,\sqcup,\sqcap,\odot,\Rightarrow,(0,1),(1,0))$ is a bounded commutative residuated lattice satisfying the double negation law.
\end{theorem}

\begin{proof}
Let $a,b,c,d,e,f\in L$. Obviously, $(L^2,\sqcup,\sqcap,(0,1),(1,0))$ is a bounded lattice. We have
\begin{align*}
            (x,y)\odot(z,v) & \approx(x\cdot z,(y'\cdot v')')\approx(z\cdot x,(v'\cdot y')')\approx(z,v)\odot(x,y), \\
((x,y)\odot(z,v))\odot(t,w) & \approx(x\cdot z,(y'\cdot v')')\odot(t,w)\approx((x\cdot z)\cdot t,((y'\cdot v')\cdot w')')\approx \\
                            & \approx(x\cdot(z\cdot t),(y'\cdot(v'\cdot w'))')\approx(x,y)\odot(z\cdot t,(v'\cdot w')')\approx \\
                            & \approx(x,y)\odot((z,v)\odot(t,w)), \\
            (x,y)\odot(1,0) & \approx(x\cdot1,(y'\cdot0')')\approx(x,(y'\cdot1)')\approx(x,y'')\approx(x,y).
\end{align*}
Moreover, the following are equivalent:
\begin{align*}
& (a,b)\odot(c,d)\leq(e,f), \\
& (a\cdot c,(b'\cdot d')')\leq(e,f), \\
& a\cdot c\leq e\text{ and }f\leq(b'\cdot d')', \\
& a\cdot c\leq e\text{ and }b'\cdot d'\leq f', \\
& a\cdot c\leq e\text{ and }b'\leq d'\rightarrow f', \\
& a\leq c\rightarrow e\text{ and }(d'\rightarrow f')'\leq b, \\
& (a,b)\leq(c\rightarrow e,(d'\rightarrow f')'), \\
& (a,b)\leq(c,d)\Rightarrow(e,f).
\end{align*}
Finally, we have
\begin{align*}
 (x,y)' & \approx(x,y)\Rightarrow(0,1)\approx(x\rightarrow0,(y'\rightarrow1')')\approx(x',(y'\rightarrow0)')\approx(x',y''')\approx(x',y'), \\
(x,y)'' & \approx(x',y')'\approx(x'',y'')\approx(x,y).
\end{align*}
\end{proof}

\begin{remark}\label{rem1}
Let us note that under the assumptions of Theorem~\ref{th6}, the antitone involution $(x,y)':=(x',y')$ in the full twist-product $L^2$ as well as in $P_a(\mathbf L)$ can be derived in a natural way by $(x,y)'\approx(x,y)\Rightarrow(0,1)$ since
\[
(x,y)\Rightarrow(0,1)\approx(x\rightarrow0,(y'\rightarrow1')')\approx(x',(y'\rightarrow0)')\approx(x',y''')\approx(x',y').
\]
This does not hold if $\odot$ and $\Rightarrow$ are defined by {\rm(\ref{equ1})} and {\rm(\ref{equ2})}, respectively.
\end{remark}
 
\begin{remark}
It is worth noticing that the case when the operations $\odot$ and $\Rightarrow$ are defined by {\rm(\ref{equ5})} and {\rm(\ref{equ6})}, respectively, has an interpretation e.g.\ in {\rm MV}-algebras. Namely, an {\rm MV}-algebra is an algebra $(M,\oplus,\neg,0)$ of type $(2,1,0)$ where $(M,\oplus,0)$ is a commutative monoid, $\neg$ satisfies the identity $\neg\neg x\approx x$ and $\oplus$ and $\neg$ are related by the \L ukasiewicz axiom
\[
\neg(\neg x\oplus y)\oplus y\approx\neg(\neg y\oplus x)\oplus x.
\]
Then $(M,\vee,\wedge)$ becomes a distributive lattice where
\begin{align*}
  x\vee y & :=\neg(\neg x\oplus y)\oplus y, \\
x\wedge y & :=\neg(\neg x\vee\neg y).
\end{align*}
for all $x,y\in M$. {\rm MV}-algebras serve as an algebraic semantics of the many-valued \L ukasiewicz logics, $\oplus$ is interpreted as disjunction and $\rightarrow$ defined by $x\rightarrow y:=\neg x\oplus y$ for all $x,y\in M$ as implication. If we put $x\cdot y:=\neg(\neg x\oplus\neg y)$ for all $x,y\in M$ then $x\rightarrow y\approx\neg(x\cdot\neg y)$
 and $(M,\vee,\wedge,\cdot,\rightarrow,0,1)$ forms a bounded residuated lattice satisfying the double negation law. If we now define $\odot$ and $\Rightarrow$ on the full twist-product $M^2$ by {\rm(\ref{equ5})} and {\rm(\ref{equ6})}, respectively, we obtain
\begin{align*}
      (x,y)\odot(z,v) & \approx(x\cdot z,y\oplus v), \\
(x,y)\Rightarrow(z,v) & \approx(x\rightarrow z,\neg y\cdot v).
\end{align*}
In fact, the lattice $\mathbf L=(L,\vee,\wedge)$ from Example~\ref{ex2} is an {\rm MV}-algebra where $\neg x:=x\rightarrow0$ and $x\oplus y:=\neg(\neg x\cdot\neg y)$ for all $x,y\in L$.
\end{remark}

It was shown in \cite{Ch} for Kleene lattices and in \cite{CL} for pseudo-Kleene lattices $(L\vee,\wedge,{}')$ that there exists at most one element $a$ of $L$ satisfying $a'=a$. If such an element exists in a lattice with an antitone involution, we can prove the following result.

\begin{theorem}
Let $\mathbf L=(L,\vee,\wedge,{}')$ be a lattice with an antitone involution and $a\in L$ with $a'=a$, assume that $(P_a(\mathbf L),\sqcup,\sqcap)$ is a sublattice of $(L^2,\sqcup,\sqcap)$ and $(x,y)'=(x',y')$ for all $(x,y)\in L^2$. Then $(P_a(\mathbf L),\sqcup,\sqcap,{}')$ is a pseudo-Kleene lattice if and only if $\mathbf L$ has this property.
\end{theorem}

\begin{proof}
Let $b,c\in L$ and $(d,e),(f,g)\in P_a(\mathbf L)$. We have $(x,y)'\approx(x,y)\Rightarrow(0,1)\approx(x',y')$ as explained in Remark~\ref{rem1}. If $(P_a(\mathbf L),\sqcup,\sqcap,{}')$ is a pseudo-Kleene lattice then $(b,a),(c,a)\in P_a(\mathbf L)$ and hence
\[
(b\wedge b',a\vee a')=(b,a)\sqcap(b',a')=(b,a)\sqcap(b,a)'\leq(c,a)\sqcup(c,a)'=(c,a)\sqcup(c',a')=(c\vee c',a\wedge a'),
\]
i.e.\ $b\wedge b'\leq c\vee c'$ showing that $\mathbf L$ is a pseudo-Kleene lattice. Conversely, assume $\mathbf L$ to be a pseudo-Kleene lattice. Then $d\wedge e\leq a\leq d\vee e$ whence $d'\wedge e'\leq a'\leq d'\vee e'$, i.e.\ $d'\wedge e'\leq a\leq d'\vee e'$ which shows $(d,e)'=(d',e')\in P_a(\mathbf L)$. Hence $P_a(\mathbf L)$ is closed with respect to $'$. Finally, we have
\begin{align*}
(d,e)\sqcap(d,e)' & =(d,e)\sqcap(d',e')=(d\wedge d',e\vee e')\leq(f\vee f',g\wedge g')=(f,g)\sqcup(f',g')= \\
                  & =(f,g)\sqcup(f,g)'
\end{align*}
showing that $(P_a(\mathbf L),\sqcup,\sqcap,{}')$ is a pseudo-Kleene lattice.
\end{proof}

Our next aim is to show when $P_a(\mathbf L)$ is closed under the operation $\odot$ defined by (\ref{equ5}). We prove the following.

\begin{theorem}
Let $\mathbf L=(L,\vee,\wedge,\cdot,\rightarrow,{}',1)$ be a commutative residuated lattice with an antitone involution, let $a\in L$ be idempotent with respect to $\cdot$, $\vee$-irreducible and $\wedge$-irreducible, assume $a'\cdot a'=a'$ and define $\odot$ by {\rm(\ref{equ5})}. Then $P_a(\mathbf L)$ is closed with respect to $\odot$.
\end{theorem}

\begin{proof}
Let $(b,c),(d,e)\in P_a(\mathbf L)$. We have
\[
(x,y)\odot(z,v)\approx(x\cdot z,(y'\cdot v')')\approx(z\cdot x,(v'\cdot y')')\approx(z,v)\odot(x,y).
\]
According to Theorem~\ref{th1},
\[
P_a(\mathbf L)=\{(x,y)\in L^2\mid(x,y)\text{ is comparable with }(a,a)\}.
\]
In the following we often use (i) and (ii) of Proposition~\ref{prop1}.
\begin{enumerate}[(i)]
\item Assume $(b,c)\leq(a,a)$. \\
We have $b\cdot d\leq b\leq a$ and every one of the following statements implies the next one:
\begin{align*}
         a & \leq c, \\
        c' & \leq a', \\
c'\cdot e' & \leq a', \\
         a & \leq(c'\cdot e')'.
\end{align*}
This shows $(b,c)\odot(d,e)=(b\cdot d,(c'\cdot e')')\leq(a,a)$.
\item Assume $(d,e)\leq(a,a)$. \\
Then $(b,c)\odot(d,e)=(d,e)\odot(b,c)\leq(a,a)$.
\item Assume $(a,a)\leq(b,c),(d,e)$. \\
Then $c,e\leq a$ and hence $a'\leq c',e'$ whence $a'=a'\cdot a'\leq c'\cdot a'\leq c'\cdot e'$ from which we conclude $(c'\cdot e')'\leq a$. Because of $a\leq b,d$ we have $a=a\cdot a\leq b\cdot a\leq b\cdot d$. Together we obtain $(a,a)\leq(b\cdot d,(c'\cdot e')')=(b,c)\odot(d,e)$.
\end{enumerate}
\end{proof}

Unfortunately, $P_a(\mathbf L)$ is not closed under $\Rightarrow$ defined by (C2) provided $L$ in non-trivial, i.e.\ if it has more than one element.

\begin{theorem}
Let $(L,\vee,\wedge,\cdot,\rightarrow,{}',0,1)$ be a bounded commutative residuated lattice with an antitone involution and $a\in L$ and put $(x,y)\Rightarrow(z,v):=(x\rightarrow z,(y'\rightarrow v')')$ for all $(x,y),(z,v)\in L^2$. Then $P_a(\mathbf L)$ is closed with respect to $\Rightarrow$ if and only if $|L|=1$.
\end{theorem}

\begin{proof}
Assume $P_a(\mathbf L)$ to be closed with respect to $\Rightarrow$. Since $(0,a),(a,0),(a,1),(1,a)\in P_a(\mathbf L)$ we have
\begin{align*}
(0,0) & =(0,1')=(1\rightarrow0,(a'\rightarrow a')')=(1,a)\Rightarrow(0,a)\in P_a(\mathbf L), \\
(1,1) & =(1,0')=(1,(1\rightarrow0)')=(a\rightarrow a,(0'\rightarrow1')')=(a,0)\Rightarrow(a,1)\in P_a(\mathbf L)
\end{align*}
whence $a\leq0\vee0=0\leq1=1\wedge1\leq a$ and therefore $0=1$, i.e.\ $|L|=1$.
\end{proof}

Authors' addresses:

Ivan Chajda \\
Palack\'y University Olomouc \\
Faculty of Science \\
Department of Algebra and Geometry \\
17.\ listopadu 12 \\
771 46 Olomouc \\
Czech Republic \\
ivan.chajda@upol.cz

Helmut L\"anger \\
TU Wien \\
Faculty of Mathematics and Geoinformation \\
Institute of Discrete Mathematics and Geometry \\
Wiedner Hauptstra\ss e 8-10 \\
1040 Vienna \\
Austria, and \\
Palack\'y University Olomouc \\
Faculty of Science \\
Department of Algebra and Geometry \\
17.\ listopadu 12 \\
771 46 Olomouc \\
Czech Republic \\
helmut.laenger@tuwien.ac.at

\end{document}